\documentclass[12pt]{article}
\usepackage{amsmath}
\usepackage{yfonts}
\usepackage{amssymb}
\usepackage{latexsym}
\usepackage{mathrsfs}
\usepackage{exercise}
\usepackage{graphicx}
\usepackage{pdflscape} 
\usepackage{mathtools}
\usepackage{relsize}

\usepackage{bigints}

\usepackage{amsthm}


\usepackage[utf8]{inputenc}
\usepackage[english]{babel}

\usepackage[dvipsnames]{xcolor}

\usepackage{xcolor}
\usepackage{comment}
\usepackage{soul}

\def\vv{\vskip0.5truecm}

\def\bel{\begin{equation}\label}
\def\eeq{\end{equation}}
\def\ds{\displaystyle}

\def\qq
\def\setmap{\rightsquigarrow}

\def\R{\mathbb{R}}

\def\nn{\mathbb{N}}

\newtheorem{theorem}{Theorem}[section]
\newtheorem{fact}[theorem]{Fact}
\newcommand{\bfact}{\begin{fact}}
\newcommand{\efact}{\end{fact}}

\def\carby{\subset\kern-3pt\subset}

\def\comment#1{}

\def\id{\kern.2em{\rm I}\kern-.56em{\rm 1}}

\def\bal{\begin{array}{l}}
\def\bac{\begin{array}{c}}
\def\ea{\end{array}}

\def\bel{\begin{equation}\label}
\def\eeq{\end{equation}}
\newcommand{\N}{\mathbb{N}}

\def\bel{\begin{equation}\label}
\def\eeq{\end{equation}}
\def\ds{\displaystyle}
\def\x{\mathcal x}
\def\U{\mathcal U}
\def\us{\texttt{u}}
\def\v{{\bf v}}

\def\x{{x_\star}}

\def\vs{\vskip1truecm}

\def\minimize{\hbox{\rm minimize}}
\title{\begin{huge}\bf An Abstract 
Maximum Principle \\ for constrained  minimum problems
\end{huge}
}

\author{ \it Monica Motta,
  Franco Rampazzo}

\date{}

\usepackage{lipsum}

\makeatletter
\newcommand\ackname{Acknowledgements}
\if@titlepage
  \newenvironment{acknowledgements}{%
      \titlepage
      \null\vfil
      \@beginparpenalty\@lowpenalty
      \begin{center}%
        \bfseries \ackname 
        \@endparpenalty\@M
      \end{center}}%
     {\par\vfil\null\endtitlepage}
\else
  
\fi
\makeatother

 \newcommand{\bpm}{\begin{pmatrix}}
 \newcommand{\epm}{\end{pmatrix}}
\begin{document}
\maketitle
\newtheorem{defin}{Definition}[section]
\newtheorem{rem}{Remark}[section]
\newtheorem{prop}{Proposition}[section]
\newtheorem{lemma}{Lemma}[section]
\newtheorem{cor}{Corollary}[section]
\newtheorem{example}{Example}[section]
%

 
\section{Introduction} 
This article makes no claim to originality, other than, perhaps, the simple statement here called the {\it Abstract Maximum Principle}. Actually, the whole contents are strongly based on H. Sussmann's and coauthors' papers (in particular, on \cite{suss1}, \cite{susswill}, \cite{suss2}), in which, in a much more general context,  the set-separation approach is regarded as foundational for necessary conditions for minima.
So, rather than being the exposition of original material,  this paper  has mainly a pedagogical purpose.  From the Abstract Maximum Principle it  is possible to deduce  several necessary conditions for both finite dimensional minimum problems and for optimal control problems. More in general, this Principle  seems apt  to capture some consequences of the geometric and  topological idea of (possibly vector-valued) minimization in a parametrized problem.

	\section{Approximating cones }
If $X$ is a real vector space,  a subset $C\subseteq X$ is a {\it cone }
if $c\in C\iff rc\in C$ for every  non-negative  real number  $r$. A {\it convex cone} $C$
is a cone $C$ that is also a convex subset. Clearly, a subset $C$ is a convex cone if and only if it is invariant by  non-negative linear combinations, namely, if and only if,  for every positive integer $q$ and for all $q$-tuples  $(\alpha_1,\ldots, \alpha_q)\in [0,+\infty)^q$, 
$(c_1,\ldots, c_q)\in C^q$, one has $\ds \sum_{i=1}^q \alpha_1c_i \in C$.


{\bf Definition}{ \it (Directional differentiability).} Let $n$,\,$m$ be positive integers,  and consider a point  $\hat x\in\R^m$ and a real number $\delta>0$. Let $C\subseteq\R^m$ be a convex cone, and let  $F:\hat x+C\cap B_{m}(\delta)  \to\R^n$ be  a continuous function. We say that {\rm $F$ is  differentiable at $\hat x$ in the direction of $C$,  with differential $L:\R^m\to \R^n$} if 
\bel{bol}F(\hat x + c) = F(\hat x) + L  c + o(c), \quad \forall c\in   C\cap B_{m}(\delta),
\end{equation}
{where $B_{m}(\delta)$ is the closed ball of $\R^m$ with center 0 and radius $\delta$, and $o(\cdot): C\to  \R^n$ denotes any function such that $\lim_{c\to0}\frac{|o(c)|}{|c|}=0$.}  

\subsubsection{Boltyanski approximating cones}

\begin{defin}\label{approximatingcone}
Let us consider a subset $ \mathcal A\subseteq \R^n $ and a point  $ y\in \mathcal A$, and let $ \mathbf{K}$ be a convex cone in $ \R^n$. We say that  $\mathbf{K}$ is a  {\rm Boltyanski approximating cone } to $\mathcal A$
at $y$ if $$\mathbf{K}=L  C$$
where,

$i)$ for a non-negative integer $m$, $C\subseteq\R^{m} $ is a convex cone,

$ii)$
$L:\R^m\to\R^n$ is a linear mapping, 
and 

$iii)$  there exist  $\delta>0$  and  a continuous  map
$ F: C\cap B_{m}(\delta) \to \mathcal A $  such that $F(0)=y$ and $F$ is differentiable at $0$ in the direction of $C$  with differential $L$, i.e.
\bel{bol}F(c) = y + L  c + o(c), \quad \forall c\in    C\cap B_{m}(\delta).
\end{equation}
\end{defin}
For the sake of  brevity, we  will often use the shorter expression  {\it approximating cone } in place of  {\it Boltyanski approximating cone}.
{
\begin{prop}[Equivalent definition of  approximating cone] \label{eqdefappcone} Let $ {\mathcal A}\subseteq \R^n $,  $ y\in {\mathcal A}$, and let $ \mathbf{K}$ be a convex cone in $ \R^n$.

Then, $\mathbf{K}$ is a  Boltyanski  approximating cone  to ${\mathcal A}$ at y {\em if
	and only if } there exist $\varepsilon>0$  and a continuous
map $ G:   \mathbf{K}\cap B_{n}(\varepsilon) \to{\mathcal A} $, such that $G(0)=y$ and
$$
G(k)=y + k +  o(k)
$$
as $ k\to 0$,  $k\in  \mathbf{K}\cap B_{n}(\varepsilon)$.
\end{prop}
\begin{proof} The sufficiency is trivially obtained by setting $m:=n$,  $C:=\mathbf{K}$, $\delta:=\varepsilon$ $F:=G$, $L:=Id$,  so that the stated condition coincides with the definition of  approximating cone.

To prove the necessity of the condition, let $F,\delta,C,L,\mathbf{K}$ be as
in Definition \ref{approximatingcone}.
  Without loss of generality we can assume  $m\leq  n$.
	Let $M: \R^{n}\to\R^m$ be  a right inverse matrix of $L$, namely a
	linear map such that
	$L\cdot M  k=k$ for every $k\in \R^n$,  $M \mathbf{K} = C$.  
	Finally,  let us define the neighborhood  of $0$ $W:= M  B_{n}(\delta)$,  and let us consider 
	the map $G: \mathbf{K}\cap W\to {\mathcal A}$ by setting $G(k):= F(M k)$.
	Notice that $G(0)=0$, $G$ is continuous and, since $M k\in C$ for every $k\in \mathbf{K}$,
	  $G(k) = F(M  k) = y + L\cdot Mk + o(Mk) = y + k + o(k)$.
\end{proof}

\begin{rem}\label{subapp} {\rm Clearly, if  $\mathbf{K}$ is a   Boltyanski approximating cone  to $\mathcal A$
	at $y$, then every convex subcone $\hat{\mathbf{K}}\subseteq \mathbf{K}$ is   a   Boltyanski approximating cone  to $\mathcal A$
	at $y$ as well. Indeed, if $F,\delta,C,L$ are as in Definition \ref{approximatingcone}, setting $ \hat C:=\{c\in C,\,\,\,Lc\in \hat{\mathbf{K}}\}$  and letting $\hat F:  \hat C\cap B_{m}(\delta) \to \mathcal A $ be the restriction of $F$ to $\hat C$, we obtain that $\hat C$ (is a convex cone which) verifies 
	\bel{bolhat}\hat F(c) = y + L c + o(c), \quad \forall c\in B(0,\delta)\cap \hat C, \qquad \hat{\mathbf{K}}= L\hat C.
\end{equation} }\end{rem}

%
%

\subsubsection{Examples of Boltyanski approximating cones}

{\bf  Approximating cones to smooth manifolds.}\,\,Let ${\mathcal A}\subseteq \mathbb{R}^n$ be a $m$-dimensional $C^1$ submanifold, $m\leq n$, $y\in {\mathcal A}$.

{\it Any convex cone ${\mathbf{K}}$ contained in the
tangent space $T_y {\mathcal A}$ is an     approximating cone to ${\mathcal A}$ at  $y$.}  

By Remark \ref{subapp} it is sufficient to prove that this is true  in the special case when $ {\mathbf{K}}= T_y {\mathcal A}$. For this purpose, consider some  open  neighbourhood
$W\subset\R^n$ 
of $y$ and let  $\Phi :{\mathcal A}\cap W \to \R^m $ be  a coordinate  chart from  $A\cap W$ { onto the open set  $U := \Phi\left({\mathcal A}\cap W \right)$. 
It is not restrictive to assume that $\Phi(y) =0$, so that $U$ is an open neighbourhood of $0$.   Let us choose $\delta>0$ such that $U\supseteq B_m(\delta)$  and  let us define $F$ as the restriction of $\Phi^{-1}$  to $B_m(\delta)$. Setting  $C:=\R^m$, we obtain 
$$ F(c)= y + L  c + o(c), \quad \forall c\in  C\cap B_m(\delta), \quad  T_y {\mathcal A} = LC,$$ where  $L:=\ds D\Phi^{-1}(0)$. Hence   $T_y {\mathcal A} =  L C$  is an approximating cone to ${\mathcal A}$ at $y$.
}

In particular, if  the submanifold ${\mathcal A}$ is the local zero level 
of a map $( \varphi_1,...,\varphi_{k}):\R^n\to\R^k$, namely 
$$
{\mathcal A}=\{x\in\R^n :\quad \varphi_1(x) = 0,...,\varphi_k(x)=0\}
$$
for suitable $C^1$ maps $\varphi_1,...,\varphi_{k}$  ($k=n-m$)   such that { the gradients 
$$\ds \nabla\varphi_1(y),..., \nabla\varphi_k(y)$$ are linearly independent, then
the subspace
$$\begin{array}{c}
\Big\{w\in\R^n : \ \nabla\varphi_i(y)   w = 0,\,\, \forall i=1,\ldots,k \Big\} = \ker  \left(\text{\rm span}\left\{ d\varphi_i(y),  \,\,   i=1,\ldots,k \right\} \right) 
\end{array}
$$}
(is isomorphic to $T_y{\mathcal A}$ and)  is a  Boltyanski  approximating cone to ${\mathcal A}$ at
$y$.

\vskip0.4truecm 
\noindent{\bf Approximating cones to boundaries of sublevels' intersections.}

Let $k,r$ be positive integers and consider the closed subset
$$
{\mathcal A}=\Big\{x\in\R^n :\qquad \varphi_1(x) = 0,...,\varphi_k(x)=0, \ \  h_1(x) \leq 0,...,h_r(x)\leq 0 \Big\}
$$
for suitable $C^1$ maps $\varphi_1,...,\varphi_k, h_1,...,h_r$. Let $y\in {\mathcal A}$ be such that, for a possibly empty subset  $\{i_1,\dots,i_q\}\subset \{1,\dots,r\}$, one has
$h_{i_1}(y) = 0,...,h_{i_q}(y)=0$, the gradients  $$\nabla h_{i_1} (y),...,\nabla h_{i_q} (y),\nabla\varphi_1(y),..., \nabla\varphi_k(y),$$ are linearly independent, and $h_j(y)<0$ for anx $j\in
\{1,\dots,r\}\backslash \{i_1,\dots,i_q\}$.

Then, one can check that
{ \it {$$
   {\mathbf{K}}= \left\{w\in\R^n: \  \nabla \varphi_\ell(x_\star)\,w=0, \  \nabla h_{i_j}(x_\star)\, w \leq 0, \ \ell=1,\dots,k, \ i=1,\dots,q \right\}
   $$}
%
is a  Boltyanski approximating cone to ${\mathcal A}$ at
$y$.}

\vskip0.4truecm

%

\section{Transversality of cones}\label{sec:cones}
Let    $X$ be a finite-dimensional, real vector space.

\begin{defin}\label{trancon}  Let   
$ {\mathbf{K}}_1$, ${\mathbf{K}}_2\subseteq X$  be convex cones. \begin{enumerate}
\item We say that $ {\mathbf{K}}_1$ and $ {\mathbf{K}}_2 $ are {\em transversal} if $$
{\mathbf{K}}_1-{\mathbf{K}}_2:=\big\{k_1-k_2 \ \ | \ (k_1,k_2)\in {\mathbf{K}}_1\times {\mathbf{K}}_2\big\} =X ;$$
\item  we say that $ {\mathbf{K}}_1$ and $ {\mathbf{K}}_2$  are {\em strongly transversal},
if they are transversal and\linebreak $ {\mathbf{K}}_{1}\cap {\mathbf{K}}_{2} \neq\{0\}$.

\end{enumerate}
\end{defin}
\noindent
{\bf Some examples in $\R^n$}\begin{itemize}
\item In $\R$, the only nontrivial cones are  $[0,+\infty[$, $]-\infty,0]$, and $\R$. The cones $[0,+\infty[$ and $]-\infty,0]$ are not transversal, while  all the  other  pairs of cones ($\ne\{0\}$)   are strongly transversal.
\item In $\R^2$, the cones  ${\mathbf{K}}_1:=[0,+\infty[^2$ and ${\mathbf{K}}_2:=]-\infty,0]\times [0,+\infty[$  are not transversal:  indeed, $ {\mathbf{K}}_1-{\mathbf{K}}_2= [0,+\infty[\times \R $. 
\item The cones  ${\mathbf{K}}_1={\mathbf{K}}_2:=[0,+\infty[^2\subset\R^2$ are strongly transversal. Notice incidentally  that, in $\R^3$, the cones $\tilde {\mathbf{K}}_1:={\mathbf{K}}_1\times\{0\}, \tilde {\mathbf{K}}_2:={\mathbf{K}}_2\times\{0\}$ are not transversal. 
\item  The cones $\hat {\mathbf{K}}_1=\hat {\mathbf{K}}_2:=\R\times \{0\}\subset \R^2$ are not transversal.

\item In $\R^2$, the cones  ${\mathbf{K}}_1:=\R\times \{0\}$, ${\mathbf{K}}_2:=\{0\}\times\R $  are transversal but they {\it are  not} strongly transversal.
Notice that they are subspaces: actually this is the only case where transversality can differ from strong transversality,  as shown by Proposition \ref{teo2} below.
\end{itemize} 

\begin{prop}\label{teo2}  Let   $ {\mathbf{K}}_1$,  ${\mathbf{K}}_2\subseteq X $ be convex cones. Then conditions i) and ii) below are equivalent:
\begin{itemize}
\item[\rm i)]  $ {\mathbf{K}}_1 , {\mathbf{K}}_2 $ are transversal;
\item[\rm ii)] either  $ {\mathbf{K}}_1 , {\mathbf{K}}_2 $ are strongly transversal  or  both $ {\mathbf{K}}_1$ and ${\mathbf{K}}_2 $
are linear subspaces  and $ {\mathbf{K}}_1\oplus {\mathbf{K}}_2=X $.
\end{itemize}
\end{prop}

\begin{proof} Condition  i) follows from ii) by definition. To prove that i) implies ii), let us  assume that $ {\mathbf{K}}_1 , {\mathbf{K}}_2 $ are transversal but not strongly transversal, so 
$ {\mathbf{K}}_1\cap {\mathbf{K}}_2=\{0\}$.  Let us prove  that ${\mathbf{K}}_1$ and ${\mathbf{K}}_2$ are linear subspaces.  Let $ k\in {\mathbf{K}}_1 $.  By the transversality assumption we get the existence of 
$ k_1\in {\mathbf{K}}_1 $ and $ k_2\in {\mathbf{K}}_2$  such that $
k=k_2-k_1 $, so that  $ k_1+k=k_2\in {\mathbf{K}}_2 $ and 
$ k_1+k\in {\mathbf{K}}_1 $. Since $ {\mathbf{K}}_1\cap {\mathbf{K}}_2=\{0\}$  it follows that  $ k_1+k=0$, so
$-k=k_1\in {\mathbf{K}}_1 $. Hence, for every $k\in {\mathbf{K}}_1$ one has $-k\in {\mathbf{K}}_1$. Since $ {\mathbf{K}}_1 $ is a convex cone, we deduce  that $ {\mathbf{K}}_1 $ is  a linear subspace.
Of course, the same conclusion  holds for $ {\mathbf{K}}_2 $ as well. Therefore, since ${\mathbf{K}}_1$ and ${\mathbf{K}}_2$ are linear subspaces, one has  $ {\mathbf{K}}_1+{\mathbf{K}}_2 =  {\mathbf{K}}_1-{\mathbf{K}}_2=X $, $ {\mathbf{K}}_1\cap {\mathbf{K}}_2=\{0\}$, that is $
{\mathbf{K}}_1 \oplus \mathbf{K}_2=X $.
\end{proof}

\begin{rem} {\rm Let us anticipate that Proposition \ref{teo2} will be crucial in  the proof of the Abstract Maximum Principle (see Th. \ref{AMP}), where, by construction,   one of the two involved cones is not a subspace. Hence, by  having  proved that these two cones are not strongly transversal, one gets that they are  not transversal (which eventually gives the Maximum Principle).
} \end{rem}

\begin{defin}\label{polar} Let    $K\subseteq X$ be cone. The (closed) convex cone   ${\mathbf{K}}^{\bot}\subseteq {X}^*$ defined as 
$$
{\mathbf{K}}^{\bot} \doteq \{ p\in X^* :\quad  p\cdot w \leq 0 \quad \forall \;
w\in {\mathbf{K}} \}.
$$
is called the   {\rm polar cone} of ${\mathbf{K}}$.

\end{defin}

\begin{prop}   
If  ${\mathbf{K}}_1,{\mathbf{K}}_2\subseteq X$ are closed convex  cones, then
\bel{coneint1}{\mathbf{K}}_1^\bot\cap {\mathbf{K}}_2^\bot = ({\mathbf{K}}_1+{\mathbf{K}}_2)^\bot,\end{equation}
which is equivalent  (by replacing $X$, ${\mathbf{K}}_1$, and ${\mathbf{K}}_2$ with $X^*$, ${\mathbf{K}}_1^\bot$, and ${\mathbf{K}}_2^\bot$, respectively) to
\bel{coneint2}
({\mathbf{K}}_1\cap {\mathbf{K}}_2)^\bot = {\mathbf{K}}_1^\bot+{\mathbf{K}}_2^\bot.
\end{equation}

\end{prop}

\begin{proof}
Let us prove \eqref{coneint1}. Since $({\mathbf{K}}_1+ {\mathbf{K}}_2)^\bot\subseteq {\mathbf{K}}_1^\bot$ and  $({\mathbf{K}}_1+ {\mathbf{K}}_2)^\bot\subseteq {\mathbf{K}}_2^\bot$,
one has
$$
{\mathbf{K}}_1^\bot\cap {\mathbf{K}}_2^\bot\supseteq({\mathbf{K}}_1+ {\mathbf{K}}_2)^\bot.
$$
Let us see that this is an  equality. Indeed, if there existed a $p\in ({\mathbf{K}}_1^\bot\cap {\mathbf{K}}_2^\bot)\setminus({\mathbf{K}}_1+ {\mathbf{K}}_2)^\bot$, $p\neq 0$, then, for some $\tilde w_1\in {\mathbf{K}}_1$ and  $\tilde w_2\in {\mathbf{K}}_2$ one would get 
$$
0\geq {{p\cdot\tilde w_1 + p\cdot\tilde w_2=}} p\cdot(\tilde w_1 + \tilde w_2) >0,
$$
a contradiction.
\end{proof}

The only  cones  in $\R$ that are not transversal are  $[0,+\infty[$, $]-\infty,0]$, so  transversality in $\R$ is characterized by a sign condition. The result below generalizes this fact  by stating that {\it transversality of two cones ${\mathbf{K}}_1$, ${\mathbf{K}}_2$ coincides with their  linear separability}.
\begin{defin} Two cones  $\mathbf{K}_1$, $\mathbf{K}_2\subseteq X$  are {\rm linearly separable} if there exists a  linear form $
p\in X^*\backslash \{0\}$ such that $$  p \cdot k_1\geq 0 \  \ \forall  k_1\in {\mathbf{K}}_1, \quad 
p \cdot k_2\leq 0 \ \ \forall k_2\in {\mathbf{K}}_2. $$ 

\end{defin}

\begin{prop}\label{teo3}  Let   ${\mathbf{K}}_1,{\mathbf{K}}_2\subseteq X$ be convex cones.  The following condititions are equivalent:
\begin{itemize}
\item ${\mathbf{K}}_1$ and ${\mathbf{K}}_2$ are not transversal; 
\item ${\mathbf{K}}_1$ and ${\mathbf{K}}_2$ are {\rm linearly separable}.
\end{itemize}
\end{prop}


{\it Proof.}  Assume that ${\mathbf{K}}_1$, ${\mathbf{K}}_2$ are not transversal, namely ${\mathbf{K}}_1- {\mathbf{K}}_2\not =X$. Notice that  ${\mathbf{K}}_1-{\mathbf{K}}_2$ is convex,
for ${\mathbf{K}}_1$ and $-{\mathbf{K}}_2$ are convex subsets. Hence there exists  a linear form $p\neq 0 $ such that
$ p \cdot k\geq 0$ for all $ k\in {\mathbf{K}}_1-{\mathbf{K}}_2$.\footnote{It is trivial to verify that  if ${\mathbf{K}}$ is a convex set different from $X$ then ${\mathbf{K}}^{\perp}$ is a non zero convex cone.} If we take $k\in {\mathbf{K}}_1$ (i.e. $k_2=0$), then $p\cdot k\geq 0$, while if we take    $k\in - {\mathbf{K}}_2$  (i.e. $k_1=0$) we obtain $p\cdot k\leq 0$.  Conversely, let   ${\mathbf{K}}_1$, $ {\mathbf{K}}_2$ be transversal,  so that   ${\mathbf{K}}_1- {\mathbf{K}}_2 = X$.  If $ p \cdot k_1\geq 0$ for all $ k_1\in {\mathbf{K}}_1$ and $ p \cdot k_2\leq 0$ for all $ k_2\in {\mathbf{K}}_2$,
then  $p \cdot k\geq 0$ for all  $ k\in {\mathbf{K}}_1-{\mathbf{K}}_2 =X$, which implies $p = 0$.

\section{Set separation and minima
}

%

\subsection{Directional Open Mapping results}

Roughly speaking   an {\it Open Mapping } result consists in  the possibility of deducing that, given a map $F:X\to Y$, the local image $F(U)$ of a neighbourhood $U$ of  $x\in X$ of is enough `thick', meaning that  it contains an open subset of $Y$. 
%
Actually, in Theorem \ref{directional}  we will deal with a {\it directional} generalization of the classical Open Mapping theorem.

\vskip0.7truecm

From the classical Inverse Map Theorem one obtains the following Open Mapping result:
\begin{theorem}[{\bf Open Mapping}]\label{omt}Let $n,m$ be positive integers, $n\leq m$, and consider a point  $\hat x\in\mathbb{R}^m$ and a real number $\delta>0$. Let $F:\hat x+B_{m}(\delta) \to\mathbb{R}^n$ be  a $C^0$ function, differentiable at $\hat x$.

	If the differential {$dF(\hat x)$}
	is surjective,
	then there exists $r>0$ such that 
	$$
	F(\hat x+B_{m}(\delta) )\supseteq {F(\hat x)+B_n(r).}  $$
	{In this case,} one says that {\rm $F$ is open at $\hat x$.
}\end{theorem}
\begin{proof}
	Up to re-ordering the components of $x=(x^1,\dots,x^m)$ and writing ${\mathbf x}^1$ and ${\mathbf x}^2$ in place of  $(x^1,\dots,x^n)$ and  $(x^{n+1},\dots,x^m) $, respectively, we can assume that 
	the matrix $\ds\frac{\partial F}{\partial{{\mathbf x^1}}}( \hat{\mathbf x}^1,\hat{\mathbf x}^2)$  is non singular.
	Therefore, by the Inverse Function Theorem, there exists $\eta\in]0,\delta]$ such that   the map
	$
	{\mathbf x^1}\mapsto F({\mathbf x^1},\hat{\mathbf x}^2)
	$ is a diffeomorphism from $\hat{\mathbf x}^1 +B_n(\eta)$ onto the open set  $F\big(\hat{\mathbf x}^1 +B_n(\eta),\hat{\mathbf x}^2\big)$.
	In particular,  $F\big(\hat{\mathbf x}^1 +B_n(\eta),\hat{\mathbf x}^2\big)$  is a neighbourhood of $F(\hat x) $, so
	we can choose  $r>0$ verifying $F\Big(\hat{\mathbf x}^1 +B_n(\eta),\hat{\mathbf x}^2\Big)\supseteq F(\hat x) +B_n(r)$. Hence
	$$
	F(\hat x+B_{m}(\delta) )\supseteq F\Big(\hat{\mathbf x}^1 + B_n(\bar
	\eta) \big),\hat{\mathbf x}^2\Big)\supseteq F(\hat x) +B_n(r),
	$$
	so the proof is concluded.

\end{proof}

This result is generalized in Theorem \ref{directional} below,
where the surjectivity of the differential is replaced by the assumption  that the image  of a conic domain via a   directional differential  has  nonempty interior. 

%
%

\begin{theorem}{\bf (Directional Open Mapping )}\label{directional}If   $n,m$ are  positive integers, $\hat x\in\R^m$,  $C\subseteq\R^m$ is a closed convex cone, and $\delta>0$, let  $$F:\hat x+C\cap B_{m}(\delta) \to\R^n$$ be  a  continuous function, differentiable at $\hat x$ in the direction of $C$   with differential $L:\R^m\to \R^n$.
Moreover, let us  assume  the existence of $v\in\R^n$ such that $v\in {\rm int}(L C).  \ \footnote{For any set $Y\subset \R^n$, we use    $\text{ co}\, Y$ and $int\,\, Y$ to denote the  convex  hull the  and interior of  $Y$, respectively. } 
$

Then,
there exist a convex cone $\Gamma\subseteq\R^n$ and $ \bar r>0$,  such that
$$
v\in {\rm int}\,\Gamma\quad\hbox{\rm and} \quad F\big(\hat x + C\cap B_{m}(\delta) \big) \supseteq  F(\hat x)+\Gamma\cap B_n({\bar r}).  $$

\end{theorem}

\begin{rem}{\rm 
Let us point out that  Theorem  \ref{omt} is recovered by setting 
$C=\R^m$ and $v=0$. Indeed, the differentiability in the direction of $\R^m$ coincides with the standard differentiability. Moreover,  the fact that  $0$ must be in the interior of the convex cone $\Gamma$  is equivalent to say  that   the latter coincides with   $\R^n$. 
}
\end{rem}
Let us prepose a technical lemma  to the proof of  Theorem \ref{directional}.

\begin{lemma}\label{close}
Let $n$ be  a positive integer. Let $\hat x\in \R^n$, $R>\rho>0$, and let $\phi:\hat x+B_n(R)\to\R^n$ be a continuous map     $\rho$-close to the identity, i.e.,$$|\phi(x)-x|\leq \rho,\quad \forall x\in \hat x +B_n(R).$$
Then,
\bel{picingr}
\phi(\hat x+B_n(R)\big)\supset	\hat x+B_n\big(R-\rho).
\end{equation}
\end{lemma}
\begin{proof} Up to the translation $x\to x-\hat x$ it is not restrictive to assume that $\hat x=0$.
Let us fix a point $y\in B_n(R-\rho) $ and let us define a map \linebreak $f_y:B_n(R)\to \R^n$  by letting
$$
f_y(x):=x - \phi(x) +y\quad \forall x \in B_n(R).
$$
By
$$
\begin{array}{c}
|f_y(x)|= |x - \phi(x) +y|\leq  |x - \phi(x)| +|y|\\
\leq	\rho + R-\rho = R 
\end{array}
$$
we deduce that $f_y (B_n(R)) \subseteq B_n(R)$. Since $f_y$ is continuous, by Brower Fixed Point Theorem we deduce the existence of a $x_y\in B_n(R)$ such that 
$$
x_y = f_y(x_y) = x_y - \phi(x_y) +y,
$$
hence  $y=\phi(x_y)\in  \phi(B_n(R))$. By the arbitrariness of $y\in \bar x+B_n(R-\rho) $ we get the thesis.
\end{proof}
\begin{proof}[Proof of Theorem \ref{directional}]
It is clearly  not restrictive to   assume that $\hat x=0$, $F(\hat x)=0$, so that the thesis reads: {\it there exist a convex cone $\Gamma\subseteq\R^n$ and $ \bar r>0$,  such that
$$
v\in {\rm int}(\Gamma)\quad\hbox{\rm and} \quad F\big(C\cap(B_{m}(\delta)) \big)\supseteq  \Gamma\cap \big(B_n(\bar r)\big).  $$}

If $v=0$, which implies  $LC=\R^n$, let us set \bel{basis1}(v_{1,r},\ldots, v_{n,r}):= r(e_1,\ldots,e_n)\qquad r>0,\end{equation} where $(e_1,\ldots,e_n)$ is the canonical basis of $\R^n$. 
If instead $v\neq 0$, let $Z\subset\R^n$ be an $(n-1)$-dimensional  subspace  such that 
$
Z + \R v =\R^n $ (which implies $Z\cap
\R v  = \{0\}.$)
Let $(z_1,\dots,z_{n-1})$ be a basis for $Z$ and define $z_n:= -z_1-z_2-\ldots-z_{n-1}$.
In this case,  we set \bel{basis2} 
(v_{1,r},\ldots, v_{n,r}):=(v+rz_1,\ldots, v+rz_n).
\end{equation}

We claim that in both cases the $n$ vectors  {\it $(v_{1,r},\ldots, v_{n,r})$  form a basis of $\R^n$.} 
Indeed, this  is straightforward for \eqref{basis1}. As for the  case \eqref{basis2},
let $a_1,\ldots,a_n$ be arbitrary real numbers such that 
$$\begin{array}{l} 0= a_1v_{1,r}+\ldots+a_n v_{n,r} = a_1(v+rz_1)+\ldots+a_n (v+rz_n)  \\ \ \ \,=
(a_1+\ldots+ a_n)v + r(a_1z_1+\dots a_nz_n).\end{array}$$
Then, by $v\notin Z$ and $a_1z_1+\dots a_nz_n\in Z$, one gets
$$
a_1+\ldots+ a_n=0,\quad 0=a_1z_1+\dots a_nz_n=(a_1-a_n)z_1+\dots+ (a_{n-1}-a_n)z_{n-1}=0,$$
which, since the vectors $z_1,\dots,z_{n-1}$ are linearly independent, implies 
$(a_1-a_n)=\ldots, (a_{n-1}-a_n) = 0$. Since $a_1+\ldots+ a_n=0$, this gives
$$a_1=\dots=a_{n-1}=a_n =0 ,$$
therefore the vectors $v_{1,r},\ldots, v_{n,r}$ are linearly independent, i.e. they form 
a basis of $\R^n$.
\begin{figure}[h]
\centering
\includegraphics[scale=0.40]{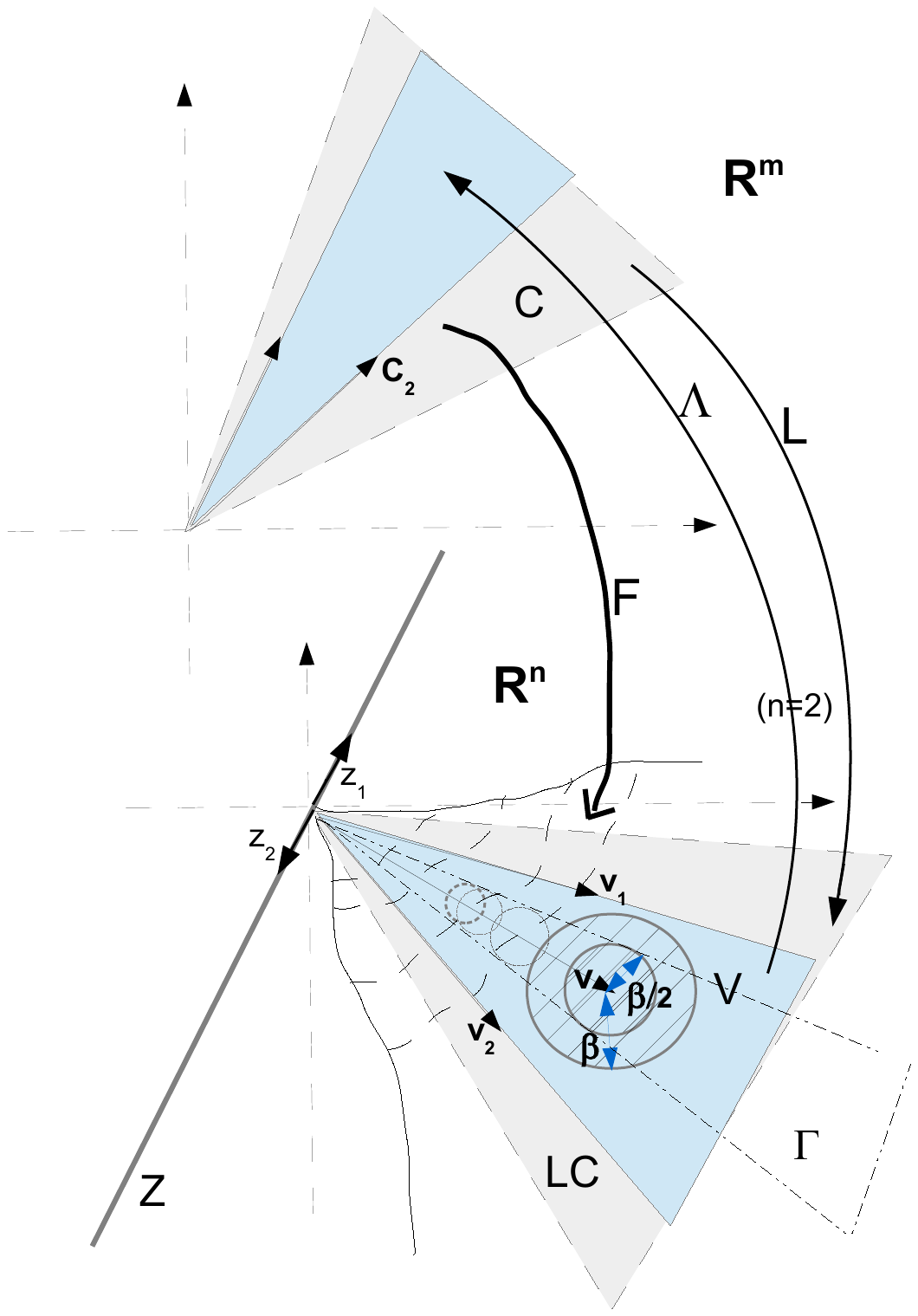}
\caption{Proof of the directional Open Mapping ($n=2$)}
\end{figure}
Now, pick $\alpha>0$ such that $v_{i,\alpha}\in L C
$ for every $i=1,\ldots,n$,\footnote{Such an  $\alpha$ does exists, for $v\in int\, LC$} and set
$\hat v_i:=v_{i,\alpha}$. Since $\hat v_i\in L C$ for every $i=1,\ldots,n$, let us choose $c_1,\ldots,c_n\in C$ such that $\hat v_i = Lc_i$,  for all $i=1,\ldots,n$.
Finally, let us define  the $m\times n$ matrix $\Lambda$  by setting $$
\Lambda \hat v_i= c_i, \qquad\forall i=1,\ldots,n.\quad \footnote{This is a good definition, for $v_1,\ldots,v_n$ is a basis.}$$ 
The matrix $\Lambda$ is    a  right {\it pseudo-inverse  of $L$}, i.e.  it verifies
\bel{pseudoi}
L\cdot \Lambda = id_{\R^m}.
\end{equation}

Let us define the closed convex cone $V:= \text{\rm span}^+\{\hat v_1,\ldots,\hat v_n\}$  ---so that $\Lambda V \subset C$--- and let us 
observe that $\ds\frac 1n\sum_{i=1}^n  \hat v_i = v$, which implies  $v\in int \, V$. In particular $int \, V\neq \emptyset$.  
Therefore there exists, $\beta>0$ such that 
$v+B_n(\beta)\subset V$. Notice that,  for every $s>0$ sufficiently small, the map
\bel{map1}
\Phi_s: v+B_n(\beta) \to\R^n, \quad \Phi_s(y):=\frac 1s F(\Lambda(sy))
\end{equation}
is well defined. Indeed, we already know that, for all $s>0$, one has \linebreak $\Lambda(s(v+B_n(\beta)))\subset C$, so that   $\Phi_s$ is well defined if and only if 
\bel{dentro1} \Lambda\big[s(v+B_n(\beta))\big]\subseteq B_n(\delta).\end{equation}
Now, if  $y\in s(v+B_n(\beta))$, one has  $|y|\leq s(|v| + \beta)$, so that, for all $x\in \Lambda\big[s(v+B_n(\beta))\big]$,  one gets  $|x|\leq 
s 
|\Lambda|(|v| + \beta)$. Therefore \eqref{dentro1} is verified  as soon as 
$$0<s\leq s_{\star}:=  \frac{\delta}{|\Lambda|(|v| + \beta)},$$
so the maps $\Phi_s$ are well defined for all $s\in]0,s_{\star}]$. 
Furthermore, 
below we will prove  the following fact: 
\vskip0.5truecm
{\bf Fact 1}. {\it There exists $\bar s \in ]0,s_{\star}]$ such that, for every $s\in]0,\bar s]$,  the map 
$\Phi_s: v+B_n(\beta)\to\R^n$ 
is a continuous  map,    $\frac{\beta}{2}$-close to the identity map.}
\vskip0.5truecm

In view of  {\bf Fact 1}, Lemma  \ref{close} tells us that, 
for every $s\in]0,\bar s]$,
$$
\big(v+B_n(\beta/2)\big)\cap V = v+B_n(\beta/2) \subseteq\Phi_s\big(v+B_n(\beta)\big)
=\Phi_s\Big(\big((v+B_n(\beta)\big)\cap V\Big) ,
$$
namely, for every $y\in v+B_n(\beta/2)\,\,  (\subset V)$  and every $s\in]0,\bar s]$,   there exists $\tilde y\in  v+B_n(\beta)\,\,  (\subset V)$
such that $y=\Phi_s(\tilde y) $. Hence  $$sy=s\Phi_s(\tilde y) =s \frac 1s F(\Lambda(s\tilde y)) \in  F(C\cap B_m(\delta)) . $$
Therefore
$$
s(v+B_n(\beta/2))\subseteq F(C\cap B_m(\delta)) , \qquad \forall s\in]0,\bar s].$$ Now  observe that $0\in  F(C\cap B_m(\delta))$  (since $F(0)=0$) and  the  set 
$$\Gamma: =\underset{s\geq 0}{\bigcup} s(v+B_n(\beta/2))$$
is a closed convex cone.
Clearly there exists $\bar r>0$ such that 
$$
B_n({\bar r})\cap \Gamma \subset \{0\}\underset{{s\in [0,\bar s]}}\bigcup s(v+B_n(\beta/2)) \subset F(C\cap B_m(\delta)) 
$$
so that the proof of Theorem \ref{directional} is concluded. It remains to prove  {\bf Fact 1}.

\vskip0.5truecm
{\bf Proof of  {\bf Fact 1}.}
Define the { \it error map } $\eta: v+B_n(\beta) \to\R^n$ by letting 
$$
\eta(y)= F(\Lambda y)-L\Lambda y\quad \forall y\in  v+B_n(\beta),
$$and let the map $\Psi:\R^{\ge0}\to\R^{\ge0}$ be defined by setting $\Psi(0)=0$ and
$$
\qquad \Psi(\gamma)= \sup \left\{\frac{|\eta(y)|}{|y|},\  |y| \leq \gamma\right\}\quad\forall\mathcal{\gamma} >0.
$$
Since $\eta(y)= o(y)$,  one has
$
\ds\lim_{\mathcal{\gamma}\to 0} \Psi(\mathcal{\gamma})= 0.
$
By the relation $L\Lambda = id_{\R^n}$, one gets 
\bel{close2}\begin{array}{c}
\left|y- \Phi_s(y)\right| =\left| y -\ds\frac 1sF(\Lambda(sy))\right|
=\\\left|y \ds-\frac 1s(L\Lambda(sy))\right| + \left|\ds\frac 1s(L\Lambda(sy)) -\ds\frac 1sF(\Lambda(sy))\right|    
\\
\leq  |y-y |+ \left|\ds\frac 1s \eta(sy)\right| = |y| \left|\ds\frac{1}{|sy|} \eta(sy)\right| \leq (|v|+\beta) \Psi( s (|v|+\beta))
\end{array} \end{equation}
so that 
$$
\left|y- \Phi_s(y)\right| \leq \frac\beta2
$$
as soon as $s\in [0,s_1]$, with $s_1$ such that $(|v|+\beta) \Psi( s (|v|+\beta))\leq \frac\beta2$ for all $s\in [0,s_1]$.
So {\bf Fact 1} is proved, which also concludes the proof of the theorem.
\end{proof}

\subsubsection{Fermat rule by set separation}
 Let us regard   Fermat's Theorem, namely  the main necessary condition for a local minimum of a smooth function $\Psi:\Omega \to\R $ defined on an open subset $\Omega \subseteq\R^n$,  as a result a result of suitable set separation.
Let us recall that a point $\bar x\in\Omega$ is a local minimum for $\Psi$ is there exists a $\delta>0$
such that $\bar x$ is a global minimum point for  the restriction $\Psi_{{|_{\bar x+B_n(\delta)}}}$, i.e. $\Psi(\bar x)\leq \Psi(x)$ for every $x\in\bar x+ B_n(\delta)$.  
\begin{theorem}[Fermat]
	If $\bar x\in\Omega $ is a local minimum point for
	$\Psi$,  then $\nabla\Psi(\bar x) = 0$.
\end{theorem} 

Let us begin with observing that we can equivalently express the notion of minimum local point in topological terms as a {\it local set-separation}. Let us define rigorously this notion:
\begin{defin} Let $X$ be a topological space  and let us consider two subsets $ {\mathcal A}_1$, $ {\mathcal A}_2\subseteq X$.  If $y\in  {\mathcal A}_1\cap  {\mathcal A}_2$, we say that $ {\mathcal A}_1$ and $ {\mathcal A}_2$ are {\em  locally separated} provided there exists a neighborhood $N$ of $y$ such that
	$$
	 {\mathcal A}_1\cap  {\mathcal A}_2\cap N = \{y\}
	$$
\end{defin}

If we  define  $graph(\Psi)\subset\R^{n+1}$
and the {\it profitable set }  $\mathcal{P}\subset\R^{n+1}$ as
$$graph(\Psi):= \{(x,\Psi(x)) \,\,\,x\in \Omega\}, \ \ \mathcal{P}=\left( \Omega\times (-\infty,\Psi(\bar x))\right)\bigcup \{(\bar x,\Psi(\bar x))\}$$
it is trivial to verify the following fact:
\vskip0,4truecm
\noindent{\bf Fact.}
{\it 
	$\bar x\in\Omega $ is a local minimum point for
	$\Psi$ if and only if the subsets  $graph(\Psi)$ and  $\mathcal{P}$ are locally separated at $(\bar x,\Psi(x))$   }
\vskip0,4truecm
So, finding necessary condition for a point to be of minimum is reduced to the problem of finding necessary conditions for two sets to be locally separated.  This will be done rigorously in the next section, but now we wish to illustrate the main idea. First, let us say that we expect that a necessary condition be given in terms of Boltyanski  approximating cones.
We will see in Corollary \ref{moltiplicatori} that the sought necessary condition is as follows:
\vskip0,4truecm
\noindent  {\bf A set-separation result} {(Corollary  \ref{moltiplicatori}).} \it Let $  {\mathcal A}_1,  {\mathcal A}_2\subset\R^q$ be locally separated at $x$ and let $ \mathbf{K}_1 $ and $ \mathbf{K}_2 $ be
		Boltyanski approximating cones at $x$ to  $ {\mathcal A}_1$ and $ {\mathcal A}_2$, respectively.
		If one of the cones 
		$ \mathbf{K}_1 $,  $\mathbf{K}_2 $ is not a linear subspace, then  $\mathbf{K}_1$ and $\mathbf{K}_2$ are {\rm linearly separable},
		namely, there exists a non-zero linear form $p\in (\R^q)^*$ such that $$ p\cdot v_1 \geq 0\,\,\,\forall v_1\in \mathbf{K}_1,\qquad p \cdot v_2 \leq 0\,\,\, \forall v_2\in \mathbf{K}_2.$$}

Let us apply this fact to our two locally separated sets,  $graph(\Psi)$ and the profitable set   $\mathcal{P}$. 
It is almost immediate to prove that the tangent space $T_{(\bar x,\Psi(\bar x))} $ to the manifold $graph(\Psi)$ is an approximating cone to $graph(\Psi)$ at $(\bar x,\Psi(\bar x))$, while the cone
$\R^n\times (-\infty,0)$ is an approximating cone to the profitable set $\mathcal{P}$  (at $(\bar x,\Psi(\bar x))$). Therefore in view of the above-stated necessary condition there must exist a  
non-zero linear form $(p,p_c)\in (\R^{n+1})^*$ such that \bel{notrav}\begin{array}{l}(p_1,p_c) \cdot(v,v^c) \geq 0\,\,\, \forall (v,v^c)\in \R^n\times (-\infty,0)\\(p_1,p_c)\cdot (w,w^c) \leq 0\,\,\,\forall  (w,w^c)\in T_{(\bar x,\Psi(\bar x))}.\end{array}\eeq 
Since $$ T_{(\bar x,\Psi(\bar x))} = span\left\{\left(\mathbf{e}_i, \frac{\partial\Psi}{\partial x^i}\right)\quad i=1,\ldots,n\right\},$$ 
is a vector space, the second relation in \eqref{notrav}  must be true as an equality, so that   that  $(p,p_c)=\alpha (-\nabla\Psi(\bar x),1)$, where $\alpha\neq 0$ can be chosen arbitrarily. For $\alpha=1$, the first relation in \eqref{notrav}  becomes
$$
(-\nabla\Psi(\bar x),1) (v,v^c) = -\nabla\Psi(\bar x) v+ v^c \geq 0\quad \forall v\in\R^n,  v^c \leq 0,
$$
which gives $-\nabla\Psi(\bar x) v\geq -v_c\geq 0,\,\,\, \forall v\in\R^n$, hence $\nabla\Psi(\bar x)=0$, which coincides  the thesis of Fermat's theorem.

We will see that also Lagrange multipliers Theorem and Kuhn Tucker Theorems can be deduced  by the above set-separation argument. Actually, the same argument will be used to prove an Abstract Maximum Principle,
 which in turn  is crucial to prove a quite general version of the Pontryagin Maximum Principle.

\subsection[Set separation and cone separability]{Set separation and  separability of approximating cones}

 The Directional Open Mapping theorem  (Theorem \ref{directional}) allows us to prove  Theorem \ref{teoteo} below, which establishes a relation between local set separation of two sets  ${\mathcal A}_1, {\mathcal A}_2$ at $x\in {\mathcal A}_1\cap {\mathcal A}_2$ and the linear separability of cones ${\mathbf{K}}_1,{\mathbf{K}}_2$, provided ${\mathbf{K}}_i$ is an  approximating cone of ${\mathcal A}_i$, $i=1,2$, at $x$.  In turn,  Theorem  \ref{teoteo} may be regarded as  the cornerstone of the proof of the Abstract Maximum Principle, where the locally separated sets 
coincide with the  `profitable set' and to  the `augmented reachable set' (see \eqref{PT}-\eqref{ARS}).

\begin{defin} Let $X$ be a topological space , and let us consider two subsets ${\mathcal A}_1$, ${\mathcal A}_2\subseteq X$.  If $y\in {\mathcal A}_1\cap {\mathcal A}_2$, we say that ${\mathcal A}_1$ and ${\mathcal A}_2$ are {\em  locally separated} provided there exists a neighborhood $N$ of $y$ such that
$$
{\mathcal A}_1\cap {\mathcal A}_2\cap N = \{y\}
$$
\end{defin}

%
%
%

\begin{theorem}\label{teoteo}
Let $ {\mathcal A}_1 $ and $ {\mathcal A}_2 $ be  subsets of $ \R^{n} $, $ x\in {\mathcal A}_1\cap {\mathcal A}_2 $, 
and $ {\mathbf{K}}_1$, ${\mathbf{K}}_2$ be  approximating cones of ${\mathcal A}_1$ and ${\mathcal A}_2$ at $x$, respectively. If  ${\mathbf{K}}_1$ and ${\mathbf{K}}_2$ are strongly transversal,
then  the sets ${\mathcal A}_1,{\mathcal A}_2$
are not locally separated. 
\end{theorem}

\begin{figure}[h]
\centering
\includegraphics[scale=0.30]{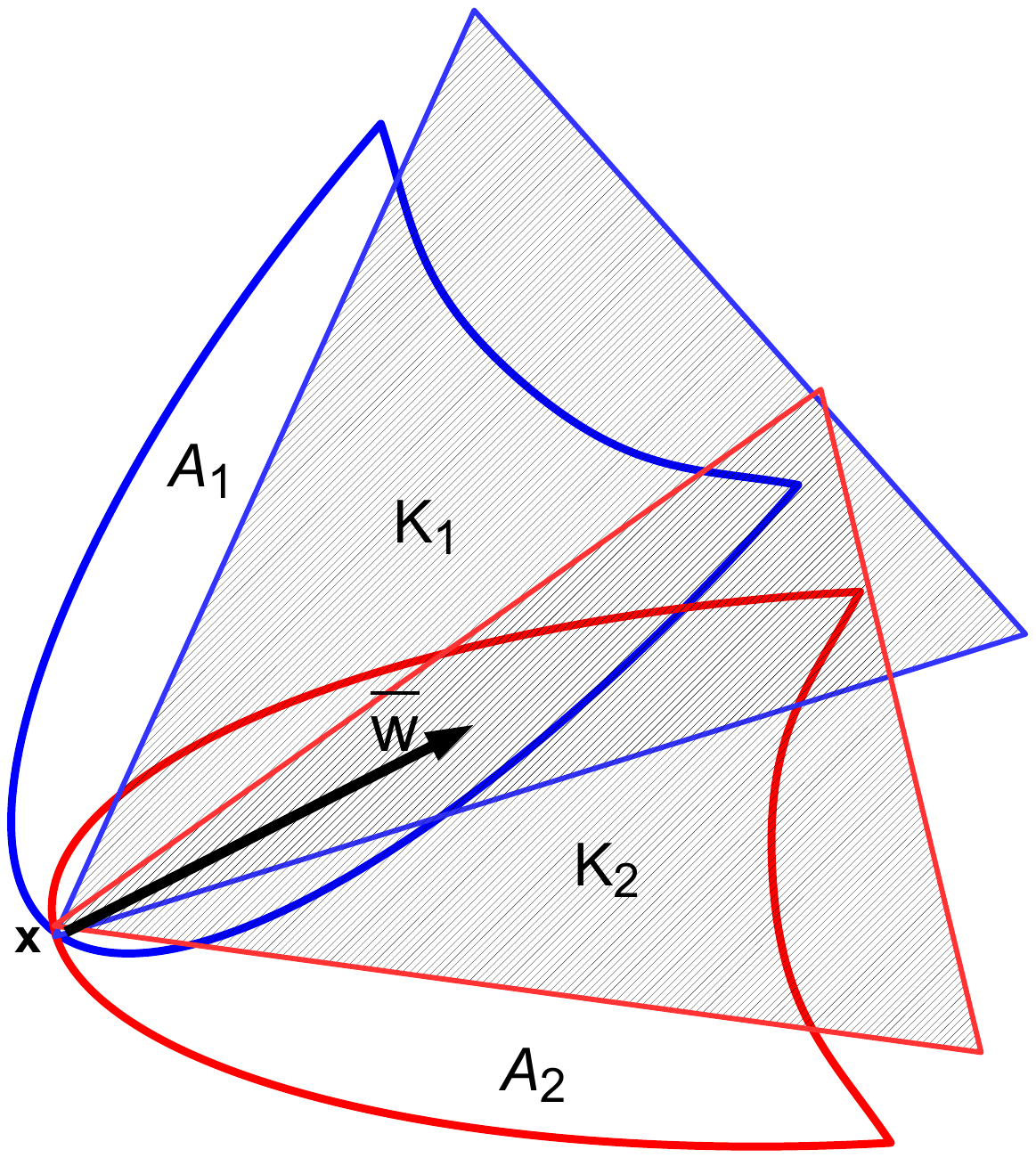}
\caption{If approximating cones at $x$ are strongly transversal the sets are not locally separated}
\end{figure}

\begin{proof} Without loss of generality we can assume $x=0$. Let $C_i, F_i,
\delta_i, L_i,m_i$, $i=1,2$, be as in the definition of Boltyanski approximating cone.   Namely, for  $i\in\{1,2\}$:\begin{itemize}

\item[ i)]  ${m_i}$ is a  positive integer and
$C_i\subset\R^{m_i}$ is a convex cone;
\item[ ii)] $L_i\in Hom(\R^{m_i},\R^n)$ and  $F_i:C_i\cap B_{m_i}(\delta_i)\to {\mathcal A}_i$ is a continuous map verifying 
$
F_i(c) = x + L_i c + o(c)$ for all {$c\in C_i\cap B_{m_i}(\delta_i)
$};
\item[ iii)] ${\mathbf{K}}_i=L_iC_i$ .
\end{itemize}

Since ${\mathbf{K}}_1$,${\mathbf{K}}_2$ are strongly transversal,  one has $ {\mathbf{K}}_1\cap {\mathbf{K}}_2\supsetneq\{0\} $. Pick  $ \bar{w}\in {\mathbf{K}}_1\cap {\mathbf{K}}_2 $ such that  $|\bar w|= 1$, and    let $
\mu :\R^n\to\R $ be  a linear function such that $ \mu (\bar{w})=1$.\footnote{For instance, if $\bar w = (\bar w^1,\ldots,\bar w^n)$,  one can  set  $\mu(v) := \sum_{i=1}^n \bar w^i v^i $, $\forall v=(v^1,\dots,v^n)\in\R^n$.}
 Let us
set 
\begin{eqnarray*}&C=C_1\times C_2\quad \delta:= \min\{\delta_1,\delta_2\}\\
 & {F: B_{m_1+m_2}(\delta)\cap C}\longrightarrow \R^{n+1}  \\
 & F(c_1,c_2):=\big(F_1(c_1)-F_2(c_2), \mu (F_1(c_1))\big).
\end{eqnarray*}
 The map  ${F}$  turns out to be differentiable  in the direction of $C$ at $(0,0)\in\R^{m_1}\times \R^{m_2}$,  with differential $L$ defined by setting 
$${L}(c_1,c_2):=(L_1c_1-L_2c_2, \mu (L_1c_1)) \quad\forall (c_1,c_2) \in C.$$
\vskip0.3truecm
 Below we shall prove the following fact:
\vskip0.3truecm
{\bf Claim 1}. {\it The vector $(0,1)\in \R^n\times \R$ belongs to { the interior of ${LC} $}.\footnote{Of course, this is equivalent to saying that $(0,\alpha)$ belongs to  the interior of ${LC} $ for every real number $\alpha>0$.}  
	}

\if{\it The vector $(0,1)\in \Omega\times \R$ belongs to {\it the interior of ${L(C)} $.}}\fi
\vskip0.5truecm

 In view of {\bf Claim 1},  the Directional Open Mapping Theorem (Theorem \ref{directional}) implies the existence of a convex cone $\hat C\subseteq\R^{n+1}$ and a real number $\bar\varepsilon>0$   such that  $(0,1)\in  int\,\,{\hat C}$ and $$\begin{array}{c} \hat C\cap B_{n+1}(\bar \varepsilon)\subseteq {F\big(B_{m_1+m_2}(\delta)\cap C\big)} \subset {F\Big((B_{m_1}(\delta_1)\cap C_1)\times (B_{m_2}(\delta_2)\cap C_2)\Big)}
 \end{array}$$ 
 In particular, for every   $ 0<\varepsilon\leq \bar{\varepsilon}$,  there exist $c_1^\varepsilon\in B_{m_1}(\delta_1)\cap C_1$ and $ c_2^\varepsilon\in B_{m_2}(\delta_2)\cap C_2 $ such that {\rm(}$F_1(c_1^\varepsilon)\in {\mathcal A}_1$ , $F_2(c_2^\varepsilon)\in {\mathcal A}_2$ and{\rm)}
$                                                            
F_1(c_1^\varepsilon)-F_2(c_2^\varepsilon)=0$,                
namely                                                         
\bel{c1}                                                       
{y^\varepsilon:=}F_1(c_1^\varepsilon)=F_2(c_2^\varepsilon)\in {\mathcal A}_1\cap {\mathcal A}_2 \end{equation} 
and $$\mu (F_1(c_1^\varepsilon))=\varepsilon.$$ 
{Since $\mu$ is linear,  the last relation implies  that}               
\bel{c3}                                                      
y^\varepsilon\neq 0 \quad \forall \varepsilon\in(0,\bar \varepsilon).
\end{equation}
Assume by contradiction that    ${\mathcal A}_1$ and ${\mathcal A}_2$ are  locally separated at $x=0$: this means that there exists $\eta>0$ such that
\bel{Asep} 
{\mathcal A}_1\cap {\mathcal A}_2\cap B_n(\eta) = \{0\}
\eeq
Yet, by the continuity of $F_1$ and by taking $\delta_1$ sufficiently small we can posit 
\bel{F1cont} 
{|y^\varepsilon| }\leq \eta/2,
\eeq
which, together {\eqref{c1}} contradicts \eqref{Asep}.  Hence   ${\mathcal A}_1$ and ${\mathcal A}_2$ are not   locally separated.

 To conclude the proof it remains to show the validity of  { \bf Claim 1}.
\vskip0.4truecm
\noindent \textit{Proof of   { \bf Claim 1}.} 

\noindent Since   $\bar{w}\in {\mathbf{K}}_1\cap {\mathbf{K}}_2 $,  there exists a pair  $(\bar c_1,\bar c_2)\in C_1\times C_2$ such that  
$\bar{w}=L_1\bar{c}_1=L_2\bar{c}_2. $
 Moreover, since the cones ${\mathbf{K}}_1$ and ${\mathbf{K}}_2$ are (strongly) transversal, for any  $ w\in \R^n $ there exists a pair $( c_1, c_2)\in C_1\times C_2$ such that  $ w=L_1c_1-L_2c_2 .$ 
Observe  that, for any $r\geq 0 $, one has $$L_1(c_1+r\bar{c}_1) \in {\mathbf{K}}_1,\quad L_2(c_2+r\bar{c}_2) \in {\mathbf{K}}_2, 
$$
$$
w\, = L_1c_1-L_2c_2  = L_1(c_1+r\bar{c}_1)-L_2(c_2+r\bar{c}_2),
$$
and
\bel{c4}\mu(L_1(c_1+r\bar{c}_1))=  \mu(L_1c_1)+r.  \end{equation}
Notice, in particular, that for any $\gamma\geq \mu(L_1c_1)$ we have $\mu(L_1(c_1+r\bar{c}_1)) = \mu(L_1(c_1))+rL_1(\bar{c}_1) =  \mu(L_1(c_1)+r=  \gamma$ provided 
 $r:= \gamma - \mu(L_1c_1)$.

 For every $i=1,\dots, n$, let  ${\bf e}_i$ be   the $i-$th element of the  canonical basis of $\R^n$ and set ${\bf e}_0 := - \sum_i {\bf e}_i$. In particular  \bel{int}0= \ds\sum_{i=0}^n\frac{{\bf e}_i}{n+1}\in int\,\,\text{co}\{{\bf e}_0,...,{\bf e}_n\}. \end{equation}

For any $i=0,\dots,n$, let us choose  ${c_1}_i\in C_1$,  ${c_2}_i\in C_2$ , $ i=0,\ldots,n$, such that ${\bf e}_i= L_1{c_1}_i-L_2{c_2}_i$, for all $i=0,\ldots,n$. Now choose 
 $r>0$ so that $\mu(L_1{c_1}_i)+r>0$ for all $i=0,\dots,n$, and set $$\gamma:=\ds\sum_{i=0}^{n} \frac{\mu(L_1{c_1}_i)+r}{n+1} =  \ds\sum_{i=0}^{n} \frac{\mu(L_1{c_1}_i)}{n+1} +r\,\, (>0).$$ 
 We get
 $$
 L({c_1}_i +r\bar c_1 ,{c_2}_i+r\bar c_2) = \left({\bf e}_i, \mu(L{c_1}_i)+r\right)\in LC
  \qquad\forall i=0,\dots,n.
  $$ 
 and
$$ (0,\gamma) = \frac{1}{n+1}\sum_{i=0}^n\big({\bf e}_i, \mu(L_1{c_1}_i)+r\big)\in int\,\, {LC}, $$ which, since  $\gamma>0$, is  equivalent 
to $(0,1)\in int\,\,{LC}$.
 \end{proof}

\begin{cor}\label{moltiplicatori} {Let   us consider two subsets ${\mathcal A}_1$, ${\mathcal A}_2\subseteq \R^n$.}
Let $ {\mathcal A}_1, {\mathcal A}_2$ be locally separated at $y$ and let $ {\mathbf{K}}_1 $ and $ {\mathbf{K}}_2 $ be
Boltyanski approximating cones at $y$ to  ${\mathcal A}_1$ and ${\mathcal A}_2$, respectively.
 If one of the cones 
 $ {\mathbf{K}}_1 $,  ${\mathbf{K}}_2 $ is not a linear subspace, then  ${\mathbf{K}}_1$ and ${\mathbf{K}}_2$ are linearly separable. 
\end{cor}


\begin{proof} By Theorem \ref{teoteo} we
deduce  that ${\mathbf{K}}_1$ and ${\mathbf{K}}_2$ are not strongly transversal. In fact,  ${\mathbf{K}}_1$ and ${\mathbf{K}}_1$ are not even transversal, as one of them is not a linear subspace (Proposition \ref{teo2}). Hence, Proposition \ref{teo3} applies and we get the thesis.\end{proof}

\vskip1truecm
\section{Abstract Maximum Principle (AMP)}
We will investigate  an abstract optimal control problem defined by a a five-tuple $(\Omega, \mathcal U,y,{\mathcal S},\Psi)$ such that:
\begin{enumerate}
	\item for some positive integer $n$, $\Omega$ is an open subset of $\R^n$;
\item
  $\mathcal U$ is any set, called  {\it control family },
while the elements $u\in  \mathcal U$ are called   {\it inputs}  or {\it controls};
\item $y=y[\cdot ]$ is an {\it input-output  map}, namely any map 
$$\bal
\ \ \ \ \ \ \ \ \ \ y:{\mathcal U}\to \Omega\\
\quad\ \,\,\,\,\,\,\,\,\,\,\,\,\,\,\,\,\,\,\,u\mapsto y[u] \,,
\ea
$$
and, for any $u\in {\mathcal U}$,  we call $y[u]\in\Omega$ {\it the output corresponding to the control $u$;}
\item   ${\mathcal S}\subset\Omega$ is any set, called  {\it target};
 
 \item 
$\Psi:\Omega\to\R $ is any map, called  {\it cost map}.
\end{enumerate}

A pair $(u,y)\in {\mathcal U}\times \Omega$ will be called an {\it input-output pair} if $y=y[u]$. Moreover, an input-output pair $(u,y[u])$  will be said {\it feasible} provided it satisfies the constraint $
y[u]\in {\mathcal S}.$  For any subfamily $\hat{\mathcal U}\subseteq{\mathcal U}$, the subset $y[\hat{\mathcal U}]$ will be called the  {\it  $\hat{\mathcal U}$-reachable  set.}
\begin{defin}\label{locmin} Let us consider a subset $\hat{\mathcal U}\subseteq  {\mathcal U}$ and  a feasible process $(u_\star,y_\star)$.
We say that the process $(u_\star,y_\star) $  is a {\rm 
$\hat{\mathcal U}$-minimizer}  of
$$
\ds
{\minimize}\,\, \Psi(y[u])
$$
 if 
\bel{minab}
\Psi(y[u]) \geq \Psi(y_\star) \end{equation} for all feasible controls   $u\in \hat{\mathcal U}$.\footnote{ I.e. controls $u\in \hat{\mathcal U}$ such that $y[u]\in  {\mathcal S}$ }
 \end{defin}

We will state and prove a necessary condition, in the form of an {\it  Abstract Maximum Principle} (AMP), for this minimum problem.
This result can be applied to a vast class of problems, from finite-dimensional problems to classical optimal control problems with end-point constraints (Section \ref{nonfreech}).

\begin{rem}{\rm Let us point out  that   the family of controls $ {\mathcal U}$ is just an abstract set, with no topological structure.  While 
$ {\mathcal U}$ is often provided with a topology, the choice of regarding it as an arbitrary set  allow us to stress the fact that the Maximum Principle is actually independent of any structural assumption on the input space. Actually, it is nothing but the expression of a certain set-separation and of the consequent separability of approximating cones.} \end{rem}

\begin{rem}{\rm In the applications to concrete minimum problems the choice of the subset  $\hat{\mathcal U}\subseteq  {\mathcal U}$ is crucial. For instance, in the case of  optimal control problems   (with or without  end-point constraints),  by choosing $\hat{\mathcal U}$ to be a  $L^1\times C^0$ neighbourhood of a process   $(u_\star,y_\star)\in  L^1\times C^0$,      we get the notion of {\it weak local minimizer}, while if $\hat{\mathcal U}$ is identified with the set of controls $u$ such that $y[u]$ belongs to a $C^0$ neighbourhood of $x_\star$ we get the notion of {\it strong local minimizer}.}
\end{rem}

\vskip0.4truecm

\begin{theorem}[{\bf AMP}]\label{AMP}
  Let  $(u_\star,y_\star)$ be a {\rm 
  	$\hat{\mathcal U}$-minimizer} and let  ${\mathcal U}_\star \subseteq \hat {\mathcal U}$ be a  subfamily of controls such that $u_\star\in {\mathcal U}_\star$ .   Furthermore,  let $\mathbf{S}$ and $\mathbf{R}$  be Boltyanski approximating cones at $ y_\star$ to the target ${\mathcal S} $ and to  the   ${\mathcal U}_\star$-reachable set $y\left[{\mathcal U}_\star\right]$, respectively.
  
Then, there exists a linear form  $(\lambda,\lambda_c)\in (\mathbf{R}^{n+1})^*$ such that $\lambda_c\leq 0$ and   the following conditions are verified:
\begin{itemize}
\item{\sc Nontriviality:}\,\,\,\,\,\,\bel{ant}(\lambda, \lambda_c)\neq 0;\end{equation}
\item {\sc Maximization:} \,\,\,\,\bel{amp11}\max \left\{\left(\lambda  + \lambda_c{\nabla}\Psi(y_\star)\right)\cdot  v, \,\,\,\,\,\, v\in \mathbf{R}\right\} =0  ;\end{equation}
\item {\sc Non-transversality:}\,\,\,\,\,\,\,\, \bel{antranv11}\lambda  \in -{\mathbf{S}}^\bot .\end{equation}

\end{itemize}

\end{theorem}
 \begin{rem}\label{lamda-1}{\rm If $\lambda=(\lambda_1,\ldots,\lambda_n)$, often the numbers $\lambda_1,\ldots,\lambda_n,\lambda_c$ are called {\it multipliers}. In particular, one says that  $\lambda_c$ is the {\it cost multiplier}.    Notice that  $(\lambda,\lambda_c)$ are determined up to multiplication by  a positive constant, so that,  without loss of generality, we might assume that $\lambda_c$ is either $-1$ or $0$. 
 		
 	}
  \end{rem}

The fact that one can choose $\lambda_c<0$   (or, equivalently,  $\lambda_c=-1$)     will be crucial in the applications {to optimal control theory}  (where we shall use the notation $p_c$ instead of $\lambda_c)$). Actually, already at this abstract level, 
the case when $\lambda_c<0$ the Maximum Principle is more informative than the case
$\lambda_c=0$. This justifies the following definition: 
\begin{defin}\label{normal}
	{A minimizer  $(u_\star,y_\star)$ for which every choice of $(\lambda^\star,\lambda_c^\star)$ in the Maximum Principle is such that 
		$\lambda_c^\star<0$ (or, equivalently,  $\lambda_c=-1$) is called {\rm a normal minimizer}. On the contrary, when it is possible to choose  $(\lambda^\star,\lambda_c^\star)$ with $\lambda_c^\star=0$, the process $(u_\star,y_\star)$  is called {\rm an abnormal minimizer}.  }
\end{defin}
\vskip4truemm
 For every $y\in \Omega$, let us set $y^c= \Psi(y)$ and   let us define the {\it profitable set } $\tilde{\mathcal S}\subset \Omega\times\R$ by setting
\bel{PT}
\tilde{\mathcal S} :=\Big\{(y, \Psi(y))\,\,: \quad y\in {\mathcal S},  \  \Psi(y) < \Psi(y_\star)\Big\}\bigcup\Big\{ \left( y_\star, \Psi(y_\star)\right)\Big\}.
\end{equation}
Moreover let us define {\it augmented   $\mathcal{U}_{\star}$-reachable  set } $\tilde {\mathcal R}\subset\Omega\times\R$
as
\bel{ARS}
\tilde{\mathcal R}: =\Big\{\left(y[u],\Psi(y[u])\right)\quad u\in {\mathcal{U}_{\star}}\Big\}
\end{equation}

and notice that 
$$\big(y_\star, \Psi(y_\star)\big)\in \tilde{\mathcal S}\cap \tilde{\mathcal R}.
$$In order to prove Theorem \ref{AMP} we will utilize the following self-evident  characterization of a local minimum:
\begin{lemma}[\bf Characterization of minima in terms of set separation]\label{min-sep}
	
	{   A feasable process $(u_\star, y_\star)$ is a local  minimum if and only if  the extended reachable set $\tilde{\mathcal R}$ and the profitable set
		$\tilde{\mathcal S}$ are locally separated at $(y_\star, \Psi(y_\star))$.}
\end{lemma}
\begin{proof}[Proof of Theorem \ref{AMP}]

We claim that {\it if $\mathbf{S}$ is a Boltyanski approximating cone to the target ${\mathcal S}$ at $y_\star$, then 
$$
\tilde{\mathbf{S}} := \mathbf{S}\times]-\infty,0] 
$$
is a Boltyanski approximating cone to $\tilde{\mathcal S}$ at  $(y_\star, \Psi(y_\star))$.}
Indeed, by hypothesis, there exist a natural  number $m\geq 0$, a convex cone   $C\subset\R^m$, a neighbourhood  $U$ of $0\in\R^m$, a homomorphism $L\in Hom(\R^m,\R^n)$,  and  a map $F: C\cap U\to {\mathcal S}$ such that  $$F(c)= y_\star + L c + o(c)\qquad\forall c\in C\cap U,$$   and $\mathbf{S}=L C$.
Consider the map $\tilde F:C\times]-\infty,0] \to\tilde{\mathcal S}$,
$$
\tilde F(c, \alpha)= (F(c),
\Psi(y_\star)-|c|^2 +\alpha)
$$
Then, $$\tilde F (c, \alpha)=\left(y_\star, \Psi(y_\star)\right)+ \tilde L\cdot(c, \alpha)+ {( o(c),-|c|^2))} = \left(y_\star, \Psi(y_\star)\right)+ \tilde L\cdot(c, \alpha)+  o(c), $$ 
 where we have set 
$$
\tilde L(c, \alpha):= (L c, \alpha)\qquad\forall (c, \alpha) \in C\times]-\infty,0],$$
so that   $\tilde{\mathbf{S}} = \tilde L\tilde C$. Therefore  $\tilde{\mathbf{S}}$  is 
an approximating cone to $\tilde{\mathcal S}$ at $(y_\star, \Psi(y_\star))$, and the claim is proved.
 
As for $\mathbf{R}$, which by hypothesis  is a Boltyanski approximating cone to the ${\mathcal U}_\star$-reachable  set $y[{\mathcal U}_\star]$ at $y_\star$, we can obviously assume it  is different from $\{0\}$.\footnote{Indeed, in the case when   $\{0\}$ is the only approximating cone at $y_\star$ to $y[{\mathcal U}_\star]$ one has $y[{\mathcal U}_\star] = \{y_\star\}$,  so that  the theorem is trivially true with $(\lambda,\lambda_c)$ = (0,-1)} Then the cone 
\bel{coneaug}\tilde{\mathbf{R}} =\Big\{(v,\nabla\Psi(y_\star)\cdot  v), \ v\in \mathbf{R}\Big\}\end{equation}
is  a Boltyanski approximating cone of  the  {\it  augmented   $\mathcal{U}_{\star}$-reachable  set } $\tilde {\mathcal R}$ at $(y_\star,\Psi(y_\star))$. Indeed, by hypothesis there exists a convex cone   $W\subset\R^m$, a neighbourhood  $U$ of $0\in\R^m$, a linear mapping $M\in Hom(\R^m,\R^n)$,  and  a map $G: W\cap U\to y[\mathcal{U}_{\star}]$ such that  $$G(w)= y_\star + M w + o(w)\qquad\forall w\in W\cap U,\quad \mathbf{R}=M W. $$  Therefore, if we define the map  $\tilde G:W\cap U\to y[\mathcal{U}_{\star}]\times \Psi\left(y[\mathcal{U}_{\star}]\right)$ and the linear mapping $\tilde M\in M\in Hom(\R^m,\R^{n+1})$ by setting, for every $w\in  W \cap U$,   $$\tilde G(w) = (G(w),\Psi\circ G(w)),\qquad \tilde M w:= \left(M w, \nabla\Psi(y_\star) (M w) \right), $$ we get 
$$\begin{array}{c}
\tilde G(w)\in \tilde{\mathcal{S}} \quad \quad \tilde MW = \tilde{\mathbf{R}}\\\tilde G(w)= (y_\star,\Psi(y_\star)) + \tilde M w + o(w)\,\,\quad\forall w\in W\cap U.
\end{array}
$$

Since the extended reachable set  $\tilde{\mathcal R}$ and  the profitable set
$\tilde{\mathcal S}$  are locally separated at $(y_\star, \Psi(y_\star))$ (Lemma \ref{min-sep}), {and $\tilde{\mathbf{S}}$ is not a subspace, by Corollary \ref{moltiplicatori}  it follows that the corresponding  approximating cones 
$\tilde{\mathbf{R}}$
 and $\tilde{\mathbf{K}}$ turn out to be  not strongly transversal. Furthermore, since $\tilde{\mathbf{S}}$ is not a subspace, Proposition \ref{teo2} implies that  they are not even transversal.  Hence, by  Proposition \ref{teo3}  $\tilde{\mathbf{R}}$
 and $\tilde{\mathbf{S}}$ }
  are linearly separable, i.e. there exists a  linear form $(\lambda,\lambda_c)\in(\R^{n+1})^*\backslash\{(0,0)\} $ such that 
  $$
 (\lambda,\lambda_c)\in -\tilde {\mathbf{S}}^\bot, \qquad   (\lambda,\lambda_c)\in \tilde{\mathbf{R}}^\bot,
 $$
 namely 
 $$
(\lambda,\lambda_c)\cdot(v,v^c) \geq 0 \quad\forall (v,v^c)\in\tilde {\mathbf{S}},\qquad  (\lambda,\lambda_c)\cdot(v,v^c) \leq 0 \quad\forall (v,v^c)\in\tilde{\mathbf{R}}.
 $$
 Since
 $
 \tilde {\mathbf{S}}^\bot = {\mathbf{S}}^\bot\times [0,+\infty[,
 $
 one gets $$\lambda\in - {\mathbf{S}}^\bot\qquad \lambda_c\leq 0.$$ Moreover,  by the definition of $\tilde{\mathbf{R}}$ one obtains 
\bel{oooo}\lambda\cdot v + \lambda_c\nabla \Psi(\hat y)\cdot  v\leq 0,\quad\forall v\in {\mathbf{R}}, \end{equation}
which is equivalent to the maximum relation \eqref{amp11}. The theorem is then proved.
\end{proof}

\vskip0.4truecm
\subsection{Finite-dimensional \ \ applications \ \ of\ \ the \ \ AMP}
We can deduce some  Calculus results  from the above Abstract Maximum Principle. 
\vskip0.4truecm
\noindent {\bf  Lagrange Multipliers Rule.} 
If the target  $\mathcal{S}$ is the local zero level 
   of a map $( \varphi_1,...,\varphi_{k}):\R^n\to\R^k$, namely 
   $$
  \mathcal{S}=\{y\in\R^n :\qquad \varphi_1(y) = 0,...,\varphi_k(y)=0\}
   $$
   for suitable $C^1$ maps $\varphi_1,...,\varphi_{n-m}$ such that the covectors
   $ \nabla\varphi_1(x),..., \nabla\varphi_k (x)$ are linearly independent in a neighborhood of $x_\star$, $ {\mathcal S}$ is a $C^1$ submanifold and  the
   the tangent space 
   $$
T_{x_\star} {\mathcal S}:= \Big\{v\in\R^n : \quad  \nabla\varphi_i(x_\star) \cdot  v = 0,\,\, i=1,\ldots,k \Big\} 
   $$
  is, in particular,  an  approximating cone to $\mathcal{S}$ at
   $x_\star$. 
   If we take ${\mathcal U}=\R^n$, $y[\cdot]= id$, i.e.  $y[u]=u$ for all $u\in {\R^n}$, and we suppose that 
   $(u_\star,x_\star) = (x_\star,x_\star) $,  ($x_\star\in {\mathcal S}$) is a locally  optimal process for the problem  
   $$\minimize  \ \Psi (x)\qquad x\in \mathcal{S},$$ by the AMP  (Theorem  \ref{AMP}) we get that there exixts $(\lambda,\lambda_c)\in (\R^{n+1})^*\backslash\{0\}$ such that 
   \bel{max1}\max \Big\{\big(\lambda  + \lambda_c\nabla\Psi (x_\star)\big)\cdot  v\leq 0, \,\,\,\forall v\in  {\mathbf{R}}\Big\} =0 \eeq and 
$$\lambda  \in -{{\mathbf{K}}}^\bot = \ds \text{\rm span}\left\{ \nabla\varphi_i(x), \ \ i=1,\dots,k\right\} .$$
 The last relation   states that  $\lambda = \ds\sum_{i=1}^k\alpha_i \nabla\varphi_i(x)$ for suitable real numbers $\alpha_1,\ldots,\alpha_k\in \R$.
 Since the reachable set  ${\mathcal R}$ coincides with $  \R^n$,  so that we can choose ${\mathbf{R}}=\R^n$,  by \eqref{max1} we get 
 $\big(\lambda  + \lambda_c\nabla \Psi(\hat y)\big) = 0$. Therefore,
\bel{lag}
 \sum_{i=1}^k\alpha_i \nabla \varphi_i(x) + \lambda_c\nabla\Psi(x_\star) = 0,
\end{equation}
a necessary conditions which coincides   with the well-known { \it Lagrange multipliers' rule.}

\vskip0.4truecm   
\noindent{{\bf Kuhn-Tucker Condition.}   Let us consider the more general  problem
$$\minimize \ \Psi (x),\qquad x\in \mathcal{S}$$ 
   $$
   \mathcal{S}=\Big\{y\in\R^n :\qquad \varphi_1(y) = 0,...,\varphi_k(y)=0,  h_1(y) \leq 0,...,h_r(y)\leq 0 \Big\}
   $$
   for suitable $C^1$ maps $\varphi_1,...,\varphi_k, h_1,...,h_r$. Let $x_\star\in \mathcal{S}$ be a local minimum, and, for a possibly empty subset  $\{i_1,\dots,i_q\}\subset \{1,\dots,r\}$, let the equalities 
 $h_{i_1}(x_\star)=0,...,h_{i_q}(x_\star)=0$ hold true (besides $\varphi_1(x_\star)=...=\varphi_k(x_\star)=0$). Moreover, assume that  the gradients  $$\ds\nabla h_{i_1} (x_\star),...,\nabla h_{i_q} (x_\star) ,\ \ds\nabla \varphi_1(x_\star),..., \nabla \varphi_k(x_\star)$$ are linearly independent, and $h_j(x_\star)<0$ for every $j\in
   \{1,\dots,r\}\backslash \{i_1,\dots,i_q\}$, so that  {
   $$
   {\mathbf{S}}= \left\{w\in\R^n: \  \nabla \varphi_\ell(x_\star)\,w=0, \  \nabla h_{i_j}(x_\star)\, w \leq 0, \ \ell=1,\dots,k, \ i=1,\dots,q \right\}
   $$
  }
  is an approximating cone to $\mathcal{S}$ at
   $x_\star$. By the Abstract Maximum Principle  we get that $\lambda  \in -{\mathbf{S}}^\bot$, so that  $$\lambda = \ds\sum_{i=1}^k\alpha_i \nabla \varphi_i(x_\star) +
   \sum_{j=1}^q  \beta_{i_j}\nabla h_{i_j}(x_\star)$$ for suitable real numbers $$\alpha_1,\ldots,\alpha_k, \,  \beta_1,\ldots,\beta_q\in \R \quad \beta_j\leq 0 \ \ \forall j=1,\ldots,q.$$
   Since, again,  the reachable set  ${\mathcal R}$ coincides with $  \R^n$,   we can choose ${\mathbf{R}}=\R^n$ as approximating cone  at $x_\star$, so that  by the maximization relation in the Maximum Principle we get
 $\ds \lambda  + \lambda_c\nabla \Psi(x_\star) = 0$. Hence
\bel{lag1}
  \ds\sum_{i=1}^k\alpha_i \nabla \varphi_i(x_\star) +
 \sum_{j=1}^q  \beta_j\nabla h_{i_j}(x_\star)+ \lambda_c\nabla \Psi(x_\star) = 0.
\end{equation}
 which coincides with the well-known { \it Kuhn-Tucker condition}.
}

 \section{Optimal control with a final target }\label{nonfreech}
{Let us now see} an infinite dimensional application of the AMP: the classical Pontryagin Maximum Principle (Theorem \ref{PMPclassic}) for the  optimal control problem 

 \begin{equation}\label{c3}
 		\ds \text{minimize}
 		 \ \left[\Psi(x(b)) + \int_{a}^{b}l(t,x(t),u(t))\,dt\right]\end{equation}
		 {over the set of processes $(u,x)$ satisfying}
 \bel{c3eq}
 		\left\{\begin{array}{l}
 			\ds\dot x(t) = f(t,x(t),u(t))\\x(a)=\bar x,\end{array} \right.
 \end{equation}
 \begin{equation}\label{c3con}
 x(b)\in {\mathcal S}.
 \end{equation}
  We shall limit ourselves to indicate the main steps of the proof, skipping technical passages like the construction of needle variations, a subject that can be recovered in every classical book on optimal control (see e.g.\cite{ces}, \cite{leem}). 
   
Our assumptions for problem \eqref{c3}-\eqref{c3con} are as follows:
\begin{itemize}
\item  the time interval $[a,b]$  is given a priori,  for some positive integers $n,m$, $U\subseteq\R^m$  is any bounded set, called the {\it control set}, and $\Omega\subseteq\R^n$ is an open subset, called the {\it state space}; 
\item the {\it controls} $u:[a,b]\to U\subseteq\R^m$ are  $L^1$-maps   and  the {\it state trajectories} $x:[a,b]\to \Omega\subseteq\R^n$ are absolutely continuous maps verifying the corresponding Cauchy problems \eqref{c3eq}.
\item The maps $f$ and $l$ are called  {\it (controlled) dynamics}  and {\it Lagrangian} or {\it current cost}, respectively. We  assume that     $(f,l)\in C^0\big([a,b]\times \Omega\times U, \R^{n+1}\big)$  and that, for every $(t,\us)\in [a,b]\times  U$, $(f,l)(t,\cdot,\us ) \in C^1(\Omega,\R^{n+1})$. 
\item The  {\it end-point cost } function $\Psi:\Omega\to\R$ is any continuously differentiable function, and the  subset  ${\mathcal S}\subset \R^n$ is called ({\it final}) {\it target}.
\end{itemize}
 If   {$u\in L^1([a,b],U)$} and there exists a   unique   solution $x\in W^{{1,1}}([a,b],\R^n)$ to the corresponding Cauchy problem \eqref{c3eq}, the pair $(u,x)$ will be called  {\it process}. Moreover, a process   $(u,x)$ is said to be  {\it feasible} provided  $x(b)\in\mathcal{S}$. 

\subsection{The Maximum Principle}

\begin{defin}\label{locmince} We shall say that a feasible process $(u_\star,x_\star)$ is a 
{\rm 
local   minimizer} of problem \eqref{c3}-\eqref{c3con} if there exists 
	 some $\delta>0$ such that
	\bel{eminoref}
	\Psi(x(b)) + \int_{a}^{b}l(t,x(t),u(t))\,dt\geq \Psi(x_\star(b)) + \int_{a}^{b}l(t,x_\star(t),u_\star(t))\,dt
\end{equation}  for all feasible  processes $(u,x)$    such that 
$$\|x-x_\star\|_{\infty} + \|u-u_\star\|_{1}  \leq \delta.$$

\end{defin}

\begin{theorem}[{\bf  Maximum Principle}]\label{PMPclassic} 
Let the feasible process 
$(u_\star,x_\star)$ be
a local minimizer  of problem \eqref{c3}-\eqref{c3con}, and let $\mathbf{S}$ be a Boltyanski approximating cone to ${\mathcal S}$ at $x_\star(b)$.
Then there exist an absolutely continuous map $ p
:[a,b]\to\left(\R^n\right)^*$, 
a real number  $p_c\leq 0$,\footnote{In fact, $p_c$ should be regarded as a linear form belonging to polar  $(\R_{\geq 0})^\bot$.} and $\lambda\in {-}{\mathbf{S}}^{\bot}$ verifying the following conditions: 
\begin{itemize}
	\item[{\rm i)}]   {\sc nontriviality condition} 
	$$
	(p,p_c) \neq (0,0)\,;
	$$
	\item[{\rm ii)} ] {\sc adjoint equation}
	
	\bel{adjcon}\begin{array}{r}\ds\dot p
		(t)=-\left(p(t)\,\frac{\partial f}{\partial x}(t,x_\star(t),u_\star(t))
		+p_c\, \frac{\partial l}{\partial x}(t,x_\star(t),u_\star(t))\right) \\ \\\hbox{ for a.e. $t\in
			[a,b]$;}
	\end{array}
\end{equation}

\item[{\rm iii)}]\noindent   {\sc maximum condition}
{\rm \begin{equation}\label{maximumc}
	\begin{array}{r}
		p(t)\cdot f(t, x_\star(t), u_\star(t))+p_c\, l(t, x_\star(t), u_\star(t))  =\qquad\qquad\\ \ds \max_{ \us\in U}  \Big(p(t) \cdot f(t, x_\star(t), \us) +p_c l\,(t, x_\star(t),\us ) \Big)\\ \\ \hbox{ for a.e. $t\in
		[a,b]$}
	\end{array}
\end{equation}}

\item[{\rm iv)}] \noindent  {\sc transversality condition}
\begin{equation}\label{eq:trasv cond}
	p(b)  = p_c \frac{\partial {\Psi} }{\partial x}(x(b)) +\lambda.
\end{equation}

\end{itemize} 
\end{theorem}\begin{rem} {\rm Let us point out that such a local minimizer is often called {\it weak}, while  a local   minimizer is said 
	{\it strong } if it   minimizes the functional among the larger set of  processes $(u,x)$  such that $x$ belongs to a $C^0$ neighborhood  of $x_\star$. Clearly, a strong local   minimizer is also a weak local   minimizer.
	
}\end{rem}
\begin{rem}\label{p-1}{\rm Since the multipliers $(p,p_c)$ are determined up to multiplication by  a positive constant, it is customary to assume that $p_c\in\{-1,0\}$.\footnote{Let us warn the reader that in most literature it is customary to  consider the non-negative  multiplier $\mu:=-p_c $ in place of $p_c$.}}  \end{rem}
\begin{rem}{\rm
	The  nontriviality condition $$(p(\cdot),p_c) \neq (0,0) \,\,\,\text{in}\,\,\,  W^{1,1}([a,b],(\R^n)^*)\times\R^*$$ is equivalent to 
	the  condition $
(p(t),p_c) \neq (0,0)\in{\R^{n+1}}\quad\forall t\in [a,b]$. Indeed, if $
(p(\bar t),p_c) = (0,0)$ for some $\bar t\in [a,b]$, then the adjoint equation \eqref{adjcon}  is linear, which implies that $p(t)=0$ for all $t\in [a,b]$.
}	\end{rem}

\vv

Agreeing with Definition \ref{normal}, let us give the definitions of normal and abnormal minimizer.

\begin{defin}
	{A {local} minimizer  $(u_\star,x_\star)$ of a general optimal control problem for which every choice of $(p^\star,p_c^\star)$ in the Maximum Principle is such that 
		$p_c^\star<0$ is called {\rm a normal minimizer}. On the contrary, when it is possible to choose  $(p,p_c^\star)\neq 0$ with $p_c^\star=0$, the process $u_\star,x_\star)$  is called  an {\rm abnormal minimizer}. Once again, thanks  to the fact that the adjoint pair $(p,p_c^\star)$ are determined up to a positive multiplicative constant, when a minimizer is normal we can assume  $p_c^\star=-1$. }
\end{defin}

\subsection{Proof of the Maximum Principle}\label{easyproof5}

{ 
	%
	We will prove  the theorem  for the Mayer problem, i.e. under the auxiliary hypothesis
	\bel{mayerh}l\equiv 0.  \footnote{This restriction is more apparent than actual, in that the general situation can be reduced to it by adding a new variable $x^{n+1}$ and a new dynamical equation $\dot x^{n+1}= l(t,x,u)$}\eeq
	 
	\vskip 0.5truecm
	
	The proof   will proceed according to the  following scheme:

	\begin{itemize}

		\item[\bf Step 1.]  Construction of   Boltyanski { approximating cones} to the reachable set at $t=b$ through the use of multiple needle variations.
		\item[\bf Step 2.] Application  of the  {Abstract Maximum Principle }  to prove a finite maximization at the last time  $t=b$ together with a non-transversality condition.\footnote{In some old literature this Step is refered to   as the {\it topological argument of the Maximum Principle})} 
		\item[\bf Step 3.] Transporting of the maximization condition   from the last instant $b$ back to a finite number of instants  $t_1,\ldots,t_m\in[a,b]$.
		
		\item[\bf Step 4.] Use of  some non-empty intersection arguments  to deduce (from { Step 3}) the maximization relation at almost every $t\in [a,b]$.
	\end{itemize}
Let us describe Steps 1-3, skipping Step 4,  which just consists in a simple topological argument.
	\vskip0.6truecm
	{\bf\large Step 1. Approximating the reachable set}
	\vskip0.6truecm
	Let us use   $\mathcal{R}$ to denote  the {\it reachable set at time $b$},  namely
	$$ 
	\mathcal{R}:= \big\{ x(b):  \,\,\,\,(x,u)\  \hbox{ process}\big\}.
	$$
	We now describe how to construct a Boltyanski approximating cone to $\mathcal{R}$ at $x_\star(b)$.
	Let   $\mathcal{L}_{u_\star}\subseteq ]a,b]$ be the subset of Lebesgue points  of the   map \linebreak $]a,b]\ni t\mapsto f(t,x_\star(t),u_\star(t))$. By Lebesgue theorem, one has
	$
	meas(\mathcal{L}_{u_\star}) = b-a
	$.
	Let $m$ be a natural number, and let us choose pairs $(t_i,\us_i)\in \mathcal{L}_{u_\star}\times U$,  $i=1,\dots,{m}$,  such that  
	$$a< t_1<t_2...
	<t_m \leq  t_{m+1}:=b.$$
	For every $ (\varepsilon_1,\dots,\varepsilon_m)\in \R_{\geq 0}^m$ sufficiently small,  let us   define the {\it multiple needle variation} $u_{\varepsilon_1,\dots,\varepsilon_m}$ by setting 
	\bel{multiple}u_{\varepsilon_1,\dots,\varepsilon_m}(t)=\left\{\begin{array}{ll} \us_i\quad &\forall t\in [t_i-\epsilon_i,t_i],\,\,i=1,\dots,m
		\\\,\\
		u_\star(t) &\forall t\in [a,b]\backslash\ds \bigcup_{i=1}^m  [t_i-\epsilon_i,t_i].
	\end{array}
	\right.
\end{equation}
The control $u_{\epsilon_1,\dots,\epsilon_m}$ can be regarded as the superposition of needle-variations at the times $t_i$ with controls $\us_i$, $i=1,\dots,m$.  
Because of our assumptions on $f$, it is easy to verify
that, for  ${(\epsilon_1,\dots,\epsilon_m)}$ sufficiently small,  the solution $x_\epsilon$ of the controlled Cauchy problem in \eqref{c3eq}  corresponding to the control $u_{\varepsilon_1,\dots,\varepsilon_m}$ does exist and is unique {on $[a,b]$}.

Moreover, the following approximation result holds true:
\begin{prop}\label{mvarfweak} There exists  $\delta>0$ such that, for every ${(\varepsilon_1,\dots,\varepsilon_m)}\in [0,\delta]^m$,
	(the solution  $x_{\varepsilon_1,\dots,\varepsilon_m}$ is defined on $[a,b]$ and) 
	{\rm \bel{mvarfweakeq}
	\begin{array}{c}
		x_{\varepsilon_1,\dots,\varepsilon_m}(b) = \\  [0.5em]
		=x_\star(b)  +\ds\sum_{i=1}^m\varepsilon_i  M(b,t_i)
		\cdot  \left[ f(t_i,x_\star(t_i),\us_i) -
		f(t_i,x_\star(t_i),u_\star(t_i))\right] \\ [0.4em]\qquad\qquad\qquad\qquad\qquad\qquad\qquad\qquad \qquad+{o(\varepsilon_1,\dots,\varepsilon_m)},
	\end{array}
\end{equation}}
where $M(\cdot,\cdot)$ is  the fundamental matrix solution associated
with the variational equation 
$
\ds\dot v(t)=   \frac{\partial f}{\partial x}(t, x_\star(t),u_\star(t)) \cdot  v(t).
$ 
\end{prop}

 Proposition \ref{mvarfweak} can be rephrased by saying that 
the  map \linebreak $({\varepsilon_1,\dots,\varepsilon_m})\mapsto x_{\varepsilon_1,\dots,\varepsilon_m}(b)$ is differentiable at
$(0,\dots,0)$ in the direction of the cone $[0,+\infty[^m$ with a differential $L:\R^m\to\R^n$ whose i-th column\footnote{$L$  is here  regarded as a $n\times m$ matrix.} is
$
M(b,t_i)
\cdot \left[f(t_i,x_\star(t_i),\us_i) -
f(t_i,x_\star(t_i),u_\star(t_i))\right],
$
for every  $i=1,\dots,m$.}

For all ${\boldsymbol{\varepsilon}}=(\varepsilon_1,\dots,\varepsilon_m)\in [0,+\infty[^m$ sufficiently small, one has 
$
x_{{\boldsymbol{\varepsilon}}}(b)\in {\mathcal{R}}.
$ Furthermore, by  Proposition \ref{mvarfweak} we get 
the following fact:
\begin{cor} The cone
\bel{varcone1}
\mathbf{R}  :=
\text{\rm span}^+
\left\{\begin{array}{l}M_{\star}(b,t_i)\cdot 
	\left[ f(t_i,x_\star(t_i),\us_i) -
	f(t_i,x_\star(t_i),u_\star(t_i))\right]\\ i=1,...,m\end{array}\right\}
\eeq
is   a Boltyanski approximating cone at $x_\star(b)$ of the reachable set $\mathcal{R}$. \end{cor}

Indeed,  Definition \ref{approximatingcone} is met by setting  $\mathcal{A}:=\mathcal{R}$, ${\mathbf{K}}=\mathbf{R}$, $x:= x_\star(b)$, $C:=[0,+\infty[^m$,  $F({\varepsilon_1,\dots,\varepsilon_m}):= x_{\varepsilon_1,\dots,\varepsilon_m}(b)$.

\vskip0.6truecm
{\bf\large Step 2. Applying the Abstract Maximum Principle}
\vskip0.6truecm
Let us define an input-output mapping mapping $\mathcal{U}\ni u\mapsto x[u]$, by setting $x[u]= x(b)\in\mathcal{R}$  as soon as $(u,x)$ is a process of the control system \eqref{c3eq}.
With this position,  the optimal control problem \eqref{c3}-\eqref{c3con} reads $$
\min_\U {\Psi}(x[u]),\qquad x[u]\in {\mathcal S},
$$ and we can apply the Abstract Maximum Principle (Theorem \ref{AMP}).
Accordingly, choosing the variational cone $\mathbf{R}$ defined in \eqref{varcone1} as an approximating cone to the reachable set $\mathcal{R}$ at $x_\star(b)$, we deduce that   there exists a linear form  $(\lambda,\lambda_c)\in (\R^{n+1})^*$ such that $\lambda_c\leq 0$ and   the following conditions are verified:

\bel{ant11}(\lambda, \lambda_c)\neq 0;\end{equation}
\bel{amp111}\max \Big\{\big(\lambda  + \lambda_c \nabla\Psi  ( x_\star(b))\big)\cdot  v, \,\,\, v\in \mathbf{R}\Big\} =0  ;\end{equation}
\bel{antranv1}\lambda  \in -\mathbf{S}^\bot .\end{equation}
Setting \bel{p(b)}
p_c:=\lambda_c,\qquad p(b) = \lambda_c \nabla\Psi( x_\star(b)) +  \lambda \qquad
\end{equation}  
from \eqref{amp111} and \eqref{varcone1} we get, for every $i=1,\ldots,m$, the  inequality
\bel{quasiadj}\begin{array}{c}
p(b)\cdot  M(b,t_i)\cdot 
\left[ f(t_i,x_\star(t_i),\us_i) -
f(t_i,x_\star(t_i),u_\star(t_i))\right] \leq 0 .
\end{array}
\end{equation}

\vskip0.6truecm\noindent
{\bf\large Step 3. Transporting cone separability  back from $b$ to $t$.}
\vskip0.6truecm

If we set   $p(t):=\ds\left(\lambda_c \nabla\Psi( x_\star(b)) +  \lambda \right)\cdot  M(b,t)$ for every $t\in[a,b]$, then the  absolutely continuous map  $p(\cdot)$ is the solution of the adjoint equation \eqref{adjcon} (remind that $l\equiv0$)  with terminal condition given by the transversality condition \eqref{eq:trasv cond}, namely $p(b):=\ds\lambda_c \nabla\Psi( x_\star(b)) +  \lambda $. Therefore, \eqref{quasiadj} can be written as
\bel{quasiadjback}\begin{array}{c}
p(t_i)\cdot 
\left[ f(t_i,x_\star(t_i),\us_i) -
f(t_i,x_\star(t_i),u_\star(t_i))\right] \leq 0 ,
\end{array}
\end{equation}
for every $i=1,\ldots,m$. Hence, we have obtained a maximum relation which coincides with  \eqref{maximumc} when the latter is restricted  to the points $t_1,\ldots,t_m$. In the next step we will utilize a infinite intersection argument to deduce \eqref{maximumc} almost everywhere in $[a,b]$.

\end{document}